\documentclass[11pt]{amsart}

\usepackage{amsfonts, amssymb, amscd}
\numberwithin{equation}{section}

\usepackage[symbol]{footmisc}

\usepackage{bm}
\usepackage{verbatim}
\usepackage{amssymb}
\usepackage{amsmath}
\usepackage{mathrsfs}
\usepackage{graphicx}
\usepackage[all,cmtip]{xy}
\usepackage[nospace,noadjust]{cite}
\usepackage{tikz-cd}
\usepackage{subcaption}
\usepackage{subfiles}
\usepackage[toc,page]{appendix}
\usepackage{mathtools}
\usepackage{comment}
\usepackage{enumerate}
\usepackage{enumitem}

\usepackage{graphicx}
\graphicspath{{images/}}

\usepackage{appendix}
\usepackage{hyperref}

\newcommand{\bb}{\bm{b}}

\newcommand{\Qq}{\mathbb{Q}}

\newcommand{\Rr}{\mathbb{R}}

\newcommand{\Nn}{\mathbb{N}}

\newcommand{\Span}{\operatorname{Span}}

\newcommand{\mld}{{\rm{mld}}}
\newcommand{\tmld}{{\rm{tmld}}}

\newcommand{\codim}{{\rm{codim}}}

\newcommand{\lct}{\operatorname{lct}}
\newcommand{\LCT}{\operatorname{LCT}}

\newcommand{\Supp}{\operatorname{Supp}}
\newcommand{\Diff}{\operatorname{Diff}}
\newcommand{\mult}{\operatorname{mult}}

\newcommand{\Ii}{\mathcal{I}}

\newcommand\MLD{{\rm{MLD}}}

\newtheorem{thm}{Theorem}[section]
\newtheorem{conj}[thm]{Conjecture}
\newtheorem{cor}[thm]{Corollary}
\newtheorem{lem}[thm]{Lemma}
\newtheorem{defn}[thm]{Definition}
\newtheorem{prop}[thm]{Proposition}
\newtheorem{ques}[thm]{Question}
\newtheorem{claim}[thm]{Claim}
\theoremstyle{definition}
\newtheorem{rem}[thm]{Remark}

\newtheorem{ex}[thm]{Example}

\theoremstyle{definition}

\begin{document}

\title{Towards the equivalence of the ACC for $a$-log canonical thresholds and the ACC for minimal log discrepancies}


\author{Jihao Liu}

\address{Department of Mathematics, The University of Uath, Salt Lake City, UT 84112, USA}
\email{jliu@math.utah.edu}

\begin{abstract}
	In this paper, we show that Shokurov's conjectures on the ACC for $a$-lc thresholds and the ACC for minimal log discrepancies are equivalent in the interval $[0,1)$. That is, the conjecture on ACC for $a$-lc thresholds holds for every $0\leq a<1$ if and only if the set of minimal log discrepancies for pairs with DCC coefficients do not have an accumulation point from below which belongs to $[0,1)$.
\end{abstract}
\date{\today}
\maketitle
\pagestyle{myheadings}\markboth{\hfill  Jihao Liu\hfill}{\hfill Towards the equivalence of the ACC for $a$-lcts and the ACC for mlds\hfill}

\tableofcontents

\section{Introduction}

In this paper we work over the field of complex numbers $\mathbb C$.

In birational geometry, many algebraic invariants are introduced to study the behavior of singularities. In this paper, we focus on two important algebraic invariants: the \emph{minimal log discrepancies} and the \emph{$a$-lc thresholds}:

\begin{defn}[Minimal log discrepancies, = Definition \ref{defn: mld}]\label{defn: mld intro}
	Let $(X\ni x,B)$ be an lc singularity. The \emph{minimal log discrepancy} of $(X\ni x,B)$ is
	$$\mld(x,X,B):=\min\{a(E,X,B)|E \text{ is a prime } \bb\text{-divisor over } X\ni x\}.$$
\end{defn}

\begin{defn}[$a$-lc thresholds, = Definition \ref{defn: alct}]\label{defn: alct intro} Let $a\geq 0$ be a real number and $(X\ni x,B)$ an lc singularity, such that $(X,B)$ is $a$-lc at $x$. The \emph{$a$-lc threshold} of $(X\ni x,B)$ with respect to an effective $\Rr$-Cartier $\mathbb R$-divisor $G$ is
	$$a\text{-}\lct_{x}(X,B;G):=\sup\{c\geq 0|(X,B+cG) \text{ is } a\text{-lc at }x\}.$$ 
\end{defn}

\noindent\textbf{Minimal log discrepancies}. Minimal log discrepancies (mlds for short) are important invariants of singularities that play a fundamental role in higher dimensional birational geometry. They are not only invariants that characterize the singularities of varieties, but also behave nicely when running the minimal model program. In \cite{Sho04}, Shokurov proved that the conjecture on termination of flips follows from two conjectures on mlds: the lower-semicontinuity (LSC for short) conjecture for mlds (see \cite[Conjecture 2.4]{Amb99}), and the ascending chain condition (ACC for short) conjecture for mlds:

\begin{conj}[{\cite[Problem 5]{Sho88}}, ACC for mlds]\label{conj: ACC for mlds} Let $d>0$ be an integer and $\Ii\subset[0,1]$ a set which satisfies the descending chain condition (DCC). Then the set
	$$\MLD(d,\Ii):=\{\mld(x,X,B)\mid (X\ni x,B) \text{ is lc}, \dim X=d, B\in \Ii\}$$
	satisfies the ACC. Here $B\in\Ii$ means that the coefficients of $B$ belong to the set $\Ii$.
\end{conj}

It turns out that both the ACC conjecture and the LSC conjecture for mlds are very subtle problems. The ACC conjecture for mlds is only completely known in dimension $\leq 2$ (cf. \cite{Ale93}, \cite{Sho91}) and for toric varieties in any dimension (cf. \cite{Bor97}, \cite{Amb06}). When $\Ii$ is a finite set, the ACC conjecture for mlds is known for a fixed germ by Kawakita (cf. \cite{Kaw14}). In dimension $3$, we know little about the case when $\Ii$ is a DCC set. Even when $\Ii$ is a finite set, the conjecture is only known for canonical $3$-folds by using classification of $3$-fold terminal singularities (cf. \cite{Mor85},  \cite{Kaw11}, \cite{Nak16}). It is very recently known (after an early version of this paper) that there exists a real number $0\leq\delta<1$ such that the conjecture holds for $\delta$-lc $3$-folds when $\Ii=\{0\}$ (cf. \cite{Jia19}). Moreover, the LSC conjecture for mlds is also only known up to dimension 3 (cf. \cite{Amb99}, \cite{Amb06}). \vspace{2mm}

\noindent\textbf{$a$-lc thresholds}. The $a$-lc thresholds are also important algebraic invariants. When $a=0$, we get the lc thersholds, and when $a=1$, we get the canonical thresholds. It is also conjectured that $a$-lc thresholds satisfy the ACC.

\begin{conj}[ACC for $a$-lc thresholds]\label{conj: ACC for aLCTs} Let $d>0$ be an integer, $a\ge 0$ a real number, and $\Ii\subset [0,1]$ and $\Ii'\subset [0,+\infty)$ two DCC sets. Then the set 
	\begin{align*}
a\text{-}\LCT(d,\Ii,\Ii'):=\{a\text{-}\lct_x(X,B;G)|& (X\ni x,B)\text{ is } \text{lc},(X,B)\text{ is } a\text{-lc at } x,\\
&\dim X=d,B\in\Ii, G\in\Ii'\}
\end{align*}
	satisfies the ACC. 
\end{conj}

Usually, Conjecture \ref{conj: ACC for aLCTs} is considered to be weaker than Conjecture \ref{conj: ACC for mlds}, and is believed to be comparably easier to tackle. There are two reasons:
\begin{itemize}
    \item Birkar and Shokurov show that Conjecture \ref{conj: ACC for mlds} implies Conjecture \ref{conj: ACC for aLCTs} of the same dimension, for every $a\geq 0$ (cf.~\cite{BS10}). As a corollary, they show Conjecture \ref{conj: ACC for aLCTs} in dimension $\leq 2$.
    \item Hacon, M\textsuperscript{c}Kernan and Xu prove Conjecture \ref{conj: ACC for aLCTs} when $a=0$ (cf. ~\cite[Theorem 1.1]{HMX14}). Their result is usually called the \emph{ACC for lc thresholds}. 
\end{itemize}

Therefore, although Conjecture \ref{conj: ACC for aLCTs} remains open in dimension $\geq 3$, it is natural to ask whether Conjecture \ref{conj: ACC for aLCTs} implies Conjecture \ref{conj: ACC for mlds}. We hope that this could lead to further progress towards Conjecture \ref{conj: ACC for mlds}.

\begin{ques}\label{ques: equivalence}
Does Conjecture \ref{conj: ACC for aLCTs} imply Conjecture \ref{conj: ACC for mlds} (therefore they are equivalent)?
\end{ques}

The main theorem of our paper indeed gives an affirmative answer to Question \ref{ques: equivalence} for non-canonical singularities. More precisely, we have an equivalence of Conjecture \ref{conj: ACC for mlds} and Conjecture \ref{conj: ACC for aLCTs} in the following sense:

\begin{thm}\label{thm: alct mld equivalence 01}
Let $d>0$ be an integer. Then the following two statements are equivalent:
\begin{enumerate}
    \item For every DCC set $\Ii\subset [0,1)$, $\MLD(d,\Ii)$ does not have an accumulation point from below which belongs to $[0,1)$. 
        \item $a\text{-}\LCT(d,\Ii,\Ii')$ satisfies the ACC for any real number $0\leq a<1$ and any two DCC sets $\Ii\subset [0,1]$ and $\Ii'\subset [0,+\infty)$.
\end{enumerate}
\end{thm}

The following two precise statements imply Theorem \ref{thm: alct mld equivalence 01} immediately:

\begin{thm}\label{thm: alct mld equivalence strict} 
Let $d>0$ be an integer and $0<a<1$ a real number. Assume that for any two finite sets $\Ii_0\subset [0,1]$ and $\Ii_0'\subset [0,+\infty)$, there exists a sequence of real numbers $\{\alpha_i\}_{i=1}^{+\infty}\subset [0,a)$ satisfying the following.
    \begin{itemize}
    \item $a=\lim_{i\rightarrow+\infty}\alpha_i$,
    \item $\alpha_i$-$\LCT(d,\Ii_0,\Ii_0')$ satisfies the ACC for every $i$, and
    \item $a$-$\LCT(d,\Ii_0,\Ii_0')$ satisfies the ACC.
    \end{itemize}
Then for any DCC set $\Ii\subset [0,1]$, $a$ is not an accumulation point from below of $\MLD(d,\Ii)$.
\end{thm}

\begin{thm}\label{thm: mld implies alct}
Let $d>0$ be an integer and $0<a<1$ a real number. Assume that $a$ is not an accumulation point from below of $\MLD(d,\Ii)$ for any DCC set $\Ii\subset [0,1]$. Then for any two DCC sets $\Ii'\subset [0,1]$ and $\Ii''\subset [0,+\infty)$, $a$-$\LCT(d,\Ii',\Ii'')$ satisfies the ACC.
\end{thm}

Since the total log discrepancy for any pair is $\leq 1$, Theorem \ref{thm: alct mld equivalence 01} implies the following corollary:

\begin{cor}\label{cor: alct total mld equivalence}
Let $d>0$ be an integer and $\Ii\subset [0,1]$ a DCC set. Assume that Conjecture \ref{conj: ACC for aLCTs} holds for every $0\leq a<1$. Then $1$ is the only possible accumulation point of
$$\{\tmld(X,B)| (X,B)\text{ is a pair, } B\in\Ii\}$$
from below, where $\tmld(X,B)$ is the total minimal log discrepancy of $(X,B)$.
\end{cor}

In a recent paper \cite{Kaw18}, Kawakita shows that the ideal version of Conjecture \ref{conj: ACC for mlds} and Conjecture \ref{conj: ACC for aLCTs} are equivalent for any fixed klt ambient variety (cf. \cite[Theorem 4.6]{Kaw18}) by using the method of generic limits. 

Our paper, however, has a completely different approach. The central part of this paper, Section 5, is a detailed study of the structure of exceptional divisors with log discrepancies between $0$ and $1$, including a reduction to klt singularities by applying a precise inversion of adjunction (cf. Section 3) and a study on multiplicities (cf. Lemma \ref{lem: divisor innermultiplicity less than 1} and Lemma \ref{multiplicity compare}). 

Besides the study of exceptional divisors with log discrepancies between $0$ and $1$, we need to use a generalized version of Birkar's result on the existence of monotonic $n$-complements (cf. \cite{Bir19}), which is called the \emph{$(n,\Ii)$-complement}, as in Section 4. We refer the readers to \cite{HLS19} for a general theory in this direction. There is a reason why we cannot directly apply Birkar's result (see Remark \ref{rem: must r divisor} for details).

\vspace{2mm}

\noindent\textit{Structure of the paper}. In Section 2, we introduce basic notation and conventions. In Section 3, we introduce a precise inversion of adjunction, which will be used to reduce our main theorem to the case of klt singularities (see Lemma \ref{lem: lc mld to klt mld} for details). In Section 4, we introduce $(n,\Ii)$-complements, and use Birkar's result of existence of $n$-complements to show the existence of local $(n,\Ii)$-complements. In Section 5, we study the behavior of exceptional divisors with log discrepancies between $0$ and $1$. In Section 6, we give several results when assuming Conjecture \ref{conj: ACC for mlds} or Conjecture \ref{conj: ACC for aLCTs}, and prove our main theorems.
\vspace{2mm}

\noindent\textbf{Acknowledgement}. The author is grateful to his advisor Christopher D. Hacon for suggesting this problem and for his constant support and many useful discussions. After Christopher D. Hacon's suggestion, Jingjun Han also suggested this question to the author when he visited the University of Utah in November 2017. He would like to thank him for many useful discussions and comments during the preparation of the first version of the paper. Besides, he would like to thank Jingjun Han for suggesting a detailed writing of Section 3 and Christopher D. Hacon for checking the details of this section. He would like to thank Chen Jiang for many useful discussions and comments. He would like to thank Masayuki Kawakita for reading his manuscript and giving many useful comments. He would like to thank Ching-Jui Lai for pointing out a gap in the proof of Lemma \ref{lem: divisor innermultiplicity less than 1} of an early version of this paper. The author was partially supported by NSF research grants no: DMS-1300750, DMS-1265285 and by a grant from the Simons Foundation; Award Number: 256202.



\section{Notation and conventions}

We adopt the standard notation and definitions in \cite{Sho92} and \cite{KM98}, and will freely use them.

\begin{defn}[$\bb$-divisors] Let $X$ be a normal variety. A $\bb$-$\Rr$ Cartier $\bb$-divisor ($\bb$-divisor for short) over $X$ is the choice of a projective birational morphism $Y\to X$ from a normal variety and an $\Rr$-Cartier $\mathbb R$-divisor $M$ on $Y$ up to the following equivalence: another projective birational morphism $Y'\to X$ from a normal variety and an $\Rr$-Cartier $\Rr$-divisor $M'$ defines the same $\bb$-divisor if there is a common resolution $W\to Y$ and $W\to Y'$ on which the pullback of $M$ and $M'$ coincide.	

	Let $E$ be a prime $\bb$-divisor over $X$. The \emph{center} of $E$ on $X$ is the closure of its image on $X$, and is denoted by $c_X(E)$. For any (not necessarily closed) point $x\in X$, if $c_X(E)=\bar x$, we say that $E$ is over $X\ni x$. If $c_X(E)$ is not a divisor, $E$ is called \emph{exceptional} over $X$. If $c_X(E)$ is a divisor, we say that $E$ is \emph{on} $X$.
\end{defn}

\begin{defn}[Multiplicities] Let $X$ be a normal variety, $E$ a prime divisor on $X$ and $D$ an $\Rr$-divisor on $X$. We define $\mult_ED$ to be the multiplicity of $E$ along $D$. 
Let $F$ be a prime $\bb$-divisor over $X$, $B$ an $\Rr$-Cartier $\Rr$-divisor on $X$ and $\phi: Y\to X$ a birational morphism such that $F$ is on $Y$. We define $\mult_FB:=\mult_F\phi^*D$.
\end{defn}

\begin{defn}[Pairs and singularities]\label{defn: positivity}
	A \emph{pair} $(X,B)$ consists of a normal variety $X$ and an effective $\Rr$-divisor $B$ on $X$ such that $K_X+B$ is $\Rr$-Cartier.
	Let $\phi:W\to X$
	be any log resolution of $(X,B)$ and let
	$$K_W+B_W:=\phi^{*}(K_X+B).$$
	The \emph{log discrepancy} of a prime divisor $D$ on $W$ with respect to $(X,B)$ is $1-\mult_{D}B_W$ and is denoted by $a(D,X,B).$
	For any real number $a\geq 0$, we say that $(X,B)$ is lc (resp. klt, $a$-lc) if $a(D,X,B)\ge0$ (resp. $>0$, $\ge a$) for every log resolution $\phi:W\to X$ as above and every prime divisor $D$ on $W$.

	We say that $(X,B)$ is dlt if $a(E,X,B)>0$ for any exceptional prime divisor $D$ over $X$ on some log resolution $\phi:W\to X$ as above. We say that $(X,B)$ is $\Qq$-factorial if every $\Qq$-divisor on $X$ is $\Qq$-Cartier.
	
	A \emph{singularity} $(X\ni x,B)$ consists of a pair $(X,B)$ and a (\textbf{not necessarily closed}) point $x\in X$. $(X\ni x,B)$ is lc (resp. klt, $a$-lc, dlt, $\Qq$-factorial) if $(X,B)$ is lc (resp. klt, $a$-lc, dlt, $\Qq$-factorial) near $x$. We say that $(X,B)$ is klt (resp. $a$-lc) at $x$ if for every prime $\bb$-divisor $E$ over $X\ni x$, $a(E,X,B)>0$ (resp. $\geq a$).
	
	For any subvariety $V$ of $X$ and point $x\in V$, we define $\codim(x,V):=\dim V-\dim\bar x$ to be the codimension of $x$ in $V$.
	
	An \emph{extraction} $f: Y\rightarrow X$ is a birational morphism such that $Y$ is $\Qq$-factorial klt. We remark that in this paper, we allow $f$ to be small.
\end{defn}

\begin{defn}[Minimal log discrepancies]\label{defn: mld}
	Let $(X\ni x,B)$ be an lc singularity. The \emph{minimal log discrepancy} of $(X\ni x,B)$ is
	$$\mld(x,X,B):=\min\{a(E,X,B)|E \text{ is a prime } \bb\text{-divisor over } X\ni x\}.$$
	If $E$ is a prime $\bb$-divisor over $X\ni x$ such that $a(E,X,B)=\mld(x,X,B)$, we say that the minimal log discrepancy of $(X\ni x,B)$ is \emph{attained at} $E$.
	We also define
\begin{align*}
    \mld(\subset x,X,B):=\min\{a(E,X,B)|&E \text{ is a prime } \bb\text{-divisor over } X,\\
&\text{such that }c_X(E)\subset\bar x\}
\end{align*}
and
$$\tmld(X,B):=\min\{a(E,X,B)| E\text{ is a prime } \bb\text{-divisor over } X\}.$$
\end{defn}

\begin{defn}[$a$-lc thresholds]\label{defn: alct} Let $a\geq 0$ be a real number and $(X\ni x,B)$ an lc singularity such that $(X,B)$ is $a$-lc at $x$. The \emph{$a$-lc threshold} of $(X\ni x,B)$ with respect to an effective $\Rr$-Cartier $\mathbb R$-divisor $G$ is
	$$a\text{-}\lct_{x}(X,B;G):=\sup\{c\geq 0|(X,B+cG) \text{ is } a\text{-lc at }x\}.$$ 
	In particular, if $a=0$, we obtain the lc threshold at $x$. For simplicity, we will use $\lct_{x}(X,B;G)$ instead of $0$-$\lct_{x}(X,B;G)$.
	
	If $E$ is a prime $\bb$-divisor over $X\ni x$ such that $a(E,X,B)>a$ and $$a(E,X,B+a\text{-}\lct_x(X,B;G)G)=a,$$
	we say that the $a$-lc threshold of $(X,B;G)$ at $x$ is attained at $E$.
\end{defn}

\begin{ex}[$a$-lc threshold not attained at any prime $\bb$-divisor]\label{ex: no divisor attaining alct}  Let $0<a<1$ be any real number, $X:=\mathbb P^2$, $H$ a curve of degree $1$ on $X$, and $x\in H$ any closed point. Then $(X,H)$ is $1$-lc at $x$, but $(X,H)$ is not $a$-lc near $x$. Since
$$a\text{-}\lct_x(X,H;H)=0,$$
 $a$-$\lct_x(X,H;H)$ is not attained at any prime $\bb$-divisor over $X\ni x$.
\end{ex}

\begin{defn}\label{defn: DCC and ACC}
	Let $\Ii$ be a set of real numbers. We define
	\begin{itemize}
	    \item $\Ii_{+}:=\{0\}\cup\{0\leq j\leq 1|j=\sum_{p=1}^li_p|i_1,\dots,i_l\in\Ii,l\in\Nn^{+}\}$, and
	    \item $D(\Ii):=\{0\leq a\leq 1|a=\frac{m-1+f}{m},m\in\Nn^+,f\in\Ii_+\}.$
	\end{itemize}
	We say that $\Ii$ satisfies the \emph{descending chain condition} (DCC) if any decreasing sequence $a_1 \ge \cdots \ge a_k \ge\cdots$ in $\Ii$ stabilizes. We say that $\Ii$ satisfies the \emph{ascending chain condition} (ACC) if any increasing sequence in $\Ii$ stabilizes. 
	
	For any real number $a$, $a$ is called an \emph{accumulation point from below} of $\Ii$, if there exists a strictly increasing sequence $\{\alpha_i\}_{i=1}^{+\infty}\subset\Ii$, such that $a=\lim_{i\rightarrow+\infty}\alpha_i$.
	
	For any normal variety $X$ and $\Rr$-divisor $B$ on $X$, we write $B\in\Ii$ if all the coefficients of $B$ belong to $\Ii$.
\end{defn}

The next lemma is elementary. A proof can be found in \cite[4.4]{MP04}.
\begin{lem}\label{lem: di is dcc}
If $\Ii\subset [0,1]$ is a DCC set, then $D(\Ii)\subset [0,1]$ is a DCC set.
\end{lem}

\begin{defn}\label{defn: dcc and acc}
Let $d>0$ be an integer, $a\geq 0$ a real number and $\Ii\subset [0,1]$ and $\Ii'\subset [0,+\infty)$ two sets. We define
	$$\MLD(d,\Ii):=\{\mld(X\ni x,B)\mid (X\ni x,B) \text{ is lc, } \dim X=d, B\in \Ii\}$$
and
\begin{align*}
a\text{-}\LCT(d,\Ii,\Ii'):=\{a\text{-}\lct_x(X,B;G)|& (X\ni x,B)\text{ is } \text{lc},(X,B)\text{ is } a\text{-lc at } x,\\
&\dim X=d,B\in\Ii, G\in\Ii'\}.
\end{align*}
For simplicity, we will use $\LCT(d,\Ii,\Ii')$ instead of $0$-$\LCT(d,\Ii,\Ii')$.
\end{defn}

The next theorem is the well-known ACC for log canonical thresholds, which is proved by Hacon, M\textsuperscript{c}Kernan and Xu.

\begin{thm}[{\cite[Theorem 1.1]{HMX14}}]\label{thm: acc lct}
Let $d>0$ be an integer and $\Ii\subset [0,1]$ and $\Ii'\subset [0,+\infty)$ two DCC sets. Then $\LCT(d,\Ii,\Ii')$ satisfies the ACC.
\end{thm}

\begin{defn}\label{defn: consensus}
    Let $d>0$ be an integer, $a>0$ a real number and $\Ii\subset [0,1]$ and $\Ii\subset [0,+\infty)$ two sets. To simplify our following statements, we introduce the following notation.
    \begin{itemize}
        \item $\bm{M}(d,\Ii,a)$ means the following statement: $a$ is not an accumulation point of $\MLD(d,\Ii)$ from below.
        \item $\bm{L}(d,\Ii,\Ii',a)$ means the following statement: $a$-$\LCT(d,\Ii,\Ii')$ satisfies the ACC.
        \item $\bm{L}'(d,\Ii,\Ii',a)$ means the following statement: there exists a sequence of real numbers $\{\alpha_i\}_{i=1}^{+\infty}\subset [0,a)$, such that
        \begin{itemize}
            \item  $a=\lim_{i\rightarrow+\infty}\alpha_i$, 
            \item $\alpha_i$-$\LCT(d,\Ii,\Ii')$ satisfies the ACC for each $i$, and
            \item $a$-$\LCT(d,\Ii,\Ii')$ satisfies the ACC.
        \end{itemize}
    \end{itemize}
\end{defn}

\begin{rem}
Theorem \ref{thm: alct mld equivalence strict} can be restated as follows: Let $d>0$ be an integer and $0<a<1$ be a real number. Assume that $\bm{L}'(d,\Ii_0,\Ii_0',a)$ holds for any two finite sets $\Ii_0\subset [0,1]$ and $\Ii_0'\subset [0,+\infty)$, then $\bm{M}(d,\Ii,a)$ holds for every DCC set $\Ii\subset [0,1]$.
\end{rem}

%

\section{A precise inversion of adjunction} In this subsection we prove several results on precise inversion of adjunction. Most of the results follow along the same lines of the proof in \cite[Section 17]{Kol92}, but there are some small differences:
\begin{itemize}
    \item We do not need to assume the termination of flips, as the required results are already known by \cite{BCHM10} and \cite{Bir12}. Thus we may show that $\mld(\subset x,X,B)=\mld(\subset x,S,\Diff_S(B))$ under certain conditions for any point $x\in S\subset X$.
    \item We deal with dlt pairs instead of plt pairs, and
    \item We deal with $\Rr$-divisors instead of $\Qq$-divisors.
\end{itemize}
As these results may be useful for other research, we decided to write their proofs in full detail.

\begin{lem}\label{lem: klt finite valuation less 1}
Let $(X,B)$ be a klt pair, then
\begin{enumerate}
 \item there are only finitely many prime $\bb$-divisors $E$ over $X$ such that $a(E,X,B)\leq 1$, 
 \item there is an extraction $g: Y\rightarrow X$ such that $g$ extracts exactly all the prime $\bb$-divisors $E$ over $X$ such that $a(E,X,B)\leq 1$, and
 \item for any birational morphism $g': Y'\rightarrow X$ which extracts exactly all the prime $\bb$-divisors $E$ over $X$ such that $a(E,X,B)\leq 1$, suppose that $K_{Y'}+B_{Y'}=g'^*(K_X+B)$, then $(Y',B_{Y'})$ is terminal.
 \end{enumerate}
\end{lem}

\begin{proof}
Let $f: W\rightarrow X$ be a log resolution of $(X,B)$, such that 
$$K_W+B_W:=f^*(K_X+B).$$
Since $(X,B)$ is klt, all the coefficients of $B_W$ are $<1$. Suppose that $1-c$ is the maximum coefficient of $B_W$. 
Let $g: W'\rightarrow W$ be the blow-up of the strata of $\Supp B_W$ that are of codimension $\geq 2$ in $W$ and $K_{W'}+B_{W'}:=g^*(K_W+B_W).$ Possibly replacing $(W,B_W)$ with $(W',B_{W'})$ and repeating this process for at most $\lceil\frac{1}{c}\rceil$ times, we may suppose that for every prime $\bb$-divisor $E$ that is exceptional over $W$, $a(E,X,B)>1$. As there are only finitely many irreducible components of $B_W$, we deduce (1).  (2) follows from (1) and  \cite[Corollary 1.4.3]{BCHM10} and (3) follows from (1).
\end{proof}

\begin{lem}\label{lem: ajundction ld correspondence}
Let $m>0$ be an integer and $(X\ni x,B)$ a dlt singularity. Assume that 
\begin{itemize}
    \item $S$ is an irreducible component of $\lfloor B\rfloor$,
    \item $K_S+B_S:=(K_X+B)|_S$, and
    \item $E_{1,S},\dots,E_{m,S}$ are distinct prime $\bb$-divisors over $S$, such that $c_S(E_{i,S})\subset\bar x$ (resp. $\subsetneq\bar x$) for every $i$.
\end{itemize}  
Then there are prime $\bb$-divisors $E_1,\dots,E_m$ over $X$ and a log resolution $f: Y\rightarrow X$ of $(X,B)$ satisfying the following.
\begin{enumerate}
\item $c_X(E_i)\subset\bar x$ (resp. $\subsetneq\bar x$) for every $i$,
    \item $E_1,\dots,E_m$ are on $Y$,
    \item $E_i|_{S_Y}=E_i\cap S_Y=E_{i,S}$ for every $i$, where $S_Y:=f^{-1}_*S$, and
    \item $a(E_i,X,B)=a(E_{i,S},S,B_S)$ for every $i$.
\end{enumerate}
In particular, 
$$\mld(\subset x,S,B_S)\geq\mld(\subset x,X,B).$$
\end{lem}
\begin{proof}
The proof almost follows from the same lines of the proof of \cite[Theorem 17.2]{Kol92}. Let $f: Y\rightarrow X$ be a log resolution of $(X,B)$, such that 
\begin{itemize}
   \item the induced morphism $f_{S}:S_Y\rightarrow S$ is a log resolution of $(S,B_S)$, where $S_Y$ is the strict transform of $S$ on $Y$, 
   \item $E_{1,S},\dots,E_{m,S}$ are on $S_Y$,
   \item $(B_Y-S_Y)\cap S_Y=\emptyset$, where $B_Y$ is the strict transform of $B$ on $Y$, and
   \item for any exceptional divisor $F$ of $f$, if $F\cap S_Y\not=\emptyset$, then $c_{X}(F)\subset S$. In particular, $c_{X}(F)=c_X(F)\cap S$ for any exceptional divisor $F$ of $f$.
\end{itemize}
Let
$$K_Y+B_Y+\sum_i(1-a(E_i,X,B))E_i=f^*(K_X+B),$$
then 
$$K_{S_Y}+\sum_{i}(1-a(E_i,X,B))(E_i\cap S_Y)\equiv f_S^*(K_S+B_S).$$
Since $E_{1,S},\dots,E_{m,S}$ are on $S_Y$, possibly reordering indices, for every integer $1\leq i\leq m$, we may assume that $E_i\cap S_Y=E_{i,S}$. Since $$c_X(E_i)=c_X(E_i)\cap S=c_S(E_{i,m})\subset\bar x \text{ (resp. }\subsetneq\bar x\text{)}$$ $f$ and $E_1,\dots,E_m$ satisfy our requirements. Since
$$\min_{1\leq i\leq m}\{a(E_{i,S},S,B_S)\}=\min_{1\leq i\leq m}\{a(E_{i},X,B)\}\geq\mld(\subset x,X,B)$$
for any set $\{E_{i,S}\}_{i=1}^m$ of prime $\bb$-divisors over $S$ such that $c_{S}(E_{i,S})\subset\bar x$, we have $\mld(\subset x,S,B_S)\geq\mld(\subset x,X,B).$
\end{proof}

\begin{lem}\label{lem: inv of adj strong}
Let $(X,B)$ be a $\Qq$-factorial dlt pair, $S$ an irreducible component of $\lfloor B\rfloor$, $K_S+B_S:=(K_X+B)|_S$ and $s\in S\subset X$ a point. Assume that 
$$a:=\min\{a(E,X,B)|c_X(E)\cap S\subset\bar s\}\leq 1,$$
then
$$a=\mld(\subset s,S,B_S)=\mld(\subset s,X,B).$$
\end{lem}

\begin{proof}
The proof is similar to \cite[Corollary 17.11(1)]{Kol92}. Since $(X,B)$ is $\Qq$-factorial dlt and $a\leq 1$, there exists an extraction $f: Y\rightarrow X$ of a prime $\bb$-divisor $E$ over $X$, such that $c_X(E)\cap S\subset\bar s$ and $a(E,X,B)=a$. Let 
\begin{itemize}
    \item $S_Y$ be the strict transform of $S$ on $Y$,
    \item $K_{S_Y}+B_{S_Y}:=f^*(K_X+B)|_{S_Y}$, and
    \item $f_S:S_Y\rightarrow S$ the birational morphism induced by $f$.
\end{itemize} 
Then $S_Y\cap E\not=\emptyset$ and $f_S(S_Y\cap E)\subset\bar s$. By adjunction, there is an irreducible component $E_{S,Y}$ of $S_Y\cap E$, two integers $m,k>0$ and a real number $c\geq 0$, such that 
$$\mult_{E_{S_Y}}B_{S_Y}=\frac{m-1+c+k(1-a)}{m}\geq 1-a,$$
which implies that $a(E_{S_Y},S_Y,B_{S_Y})\leq a.$ Thus
\begin{align*}
    \mld(\subset s,S,B_S)&\leq a(E_{S_Y},S,B_S)=a(E_{S_Y},S_Y,B_{S_Y})\leq a\\
    &=\min\{a(E,X,B)|c_X(E)\cap S\subset\bar s\}\leq\mld(\subset s,X,B).
\end{align*}
By Lemma \ref{lem: ajundction ld correspondence}, $\mld(\subset s,X,B)\leq\mld(\subset s,S,B_S)$. Thus $a=\mld(\subset s,S,B_S)=\mld(\subset s,X,B)$.
\end{proof}

\begin{lem}\label{lem: terminalization near dlt}
Let $n\geq 0$ be an integer, $(X,B)$ a dlt pair, $S_1,\dots,S_n$ the irreducible components of $\lfloor B\rfloor$, $V:=X\cap_{i=1}^nS_i$ a subvariety, and $x\in V$ a point, such that
\begin{itemize}
    \item $\codim(x,V)\geq 1$,
    \item for any irreducible component $S_0\not\in\{S_1,\dots, S_n\}$ of $\lfloor B\rfloor$, $\bar x\cap S_0=\emptyset$.
\end{itemize}
Let $\mathcal{D}$ be the set of all the prime $\bb$-divisors $E$ over $X$ such that $a(E,X,B)\leq 1$ and $c_X(E)\subsetneq\bar x$. Then
\begin{enumerate} 
\item $\mathcal{D}$ is a finite set,
\item there is an extraction $f_0: Y_0\rightarrow X$ which extracts exactly all the prime $\bb$-divisors belonging to $\mathcal{D}$, and
\item for any birational morphism $f: Y\rightarrow X$ which extracts exactly all the prime $\bb$-divisors belonging to $\mathcal{D}$, suppose that $K_Y+B_Y:=f^*(K_X+B)$, then
\begin{enumerate}
\item $f$ is an isomorphism near any codimension $1$ point of $V$,
\item $(Y,B_Y)$ is dlt,
\item for every prime $\bb$-divisor $F$ over $X$ such that 
\begin{itemize}
\item $F$ is exceptional over $Y$, and
\item  $c_X(F)\subsetneq\bar x$,
\end{itemize}
we have $a(F,Y,B_Y)=a(F,X,B)>1$.
\end{enumerate}
\end{enumerate}
\end{lem}

\begin{proof}
We use induction on $n$. When $n=0$, $(X\ni x,B)$ is klt, and the lemma follows from Lemma \ref{lem: klt finite valuation less 1}.

Suppose that $n\geq 1$. First we show (1).

Let $S:=S_n$ and $K_S+B_S:=(K_X+B)|_S$. Then $(S,B_S)$ is a dlt pair, $S_1|_{S},\dots,S_{n-1}|_{S}$ are irreducible components of $B_S$, $x\in V_S:=S\cap_{i=1}^{n-1}(S_i|_{S})\cong V$ is a point such that $\codim(x,V_S)\geq 1$, and for any irreducible component $S_{0,S}\not\in\{S_1|_S,\dots,S_{n-1}|_S\}$ of $\lfloor B_S\rfloor$, $\bar x\cap S_{0,S}=\emptyset$. 

Let $\mathcal{D}_S$ be the set of all the prime $\bb$-divisors $E_S$ over $S$ such that $a(E_S,S,B_S)\leq 1$ and $c_S(E_S)\subsetneq\bar x$. By induction, $\mathcal{D}_S$ is a finite set. Suppose that $\mathcal{D}_S=\{E_{1,S},\dots,E_{m,S}\}$. Then for every $1\leq i\leq m$, $a_i:=a(E_{i,S},S,B_S)\leq 1$ and $c_S(E_{i,S})\subsetneq\bar x$. Therefore by Lemma  \ref{lem: ajundction ld correspondence}, there is a log resolution $h: Z\rightarrow X$ of $(X,B)$ and prime $\bb$-divisors $E_1,\dots,E_m$ over $X$, such that for every $1\leq i\leq m$,
\begin{itemize}
    \item $c_X(E_i)\subsetneq\bar x$,
    \item $E_i$ is on $Z$,
    \item $E_i|_{S_Z}=E_i\cap S_Z=E_{i,S}$, where $S_Z:=h^{-1}_*S$, and
    \item $a(E_i,X,B)=a_i\leq 1$.
\end{itemize}
Let $B_Z:=h^{-1}_*B$, $E_{i,Z}:=c_Z(E_i)$ for every $i$, and let $\Psi$ be the sum of all the reduced exceptional divisors of $h$ on $Z$ except $E_{1,Z},\dots,E_{m,Z}$. Then $K_Z+B_Z+\sum_{i=1}^m(1-a_i)E_{i,Z}+\Psi$ is $\Qq$-factorial dlt, and we may run a $(K_Z+B_Z+\sum_{i=1}^m(1-a_i)E_{i,Z}+\Psi)$-MMP $\phi: Z\dashrightarrow W$ over $X$. By our construction, $\phi$ contracts exactly $\Psi$. Thus the birational morphism $g: W\rightarrow X$ extracts exactly $E_1,\dots, E_m$. Moreover, since $(Z,B_Z+\sum_{i=1}^mE_{i,Z}+\Psi)$ is log smooth, $\phi$ is an isomorphism near the generic point of $E_{i,Z}\cap S_Z$ for each $i$. In particular, $c_W(E_i)\cap S_W\not=\emptyset$ for any $i$, where $S_W:=g^{-1}_*S$.

Let $K_W+B_W:=g^*(K_X+B)$, $K_{S_W}+B_{S_W}:=(K_W+B_W)|_{S_W}$ and $g|_{S_W}: S_W\rightarrow S$ the induced birational morphism. By adjunction, $g|_{S_W}$ only extracts prime $\bb$-divisors belonging to $\mathcal{D}_S$. Thus by our previous statements,  $g|_{S_W}$ extracts exactly all the prime $\bb$-divisors belonging to $\mathcal{D}_S$.  Therefore, for any prime $\bb$-divisor $F_S$ over $S$ such that $F_S$ is exceptional over $S_W$ and $c_{S}(F_S)\subsetneq\bar x$, $a(F_S,S_W,B_{S_W})=a(F_S,S,B_S)>1$.

\begin{claim}\label{claim: first extraction corresponding to S}
For any prime $\bb$-divisor $F$ over $X$ such that
\begin{itemize}
\item $F$ is exceptional over $W$,
\item $c_X(F)\subsetneq\bar x$, and
\item $c_W(F)\cap V_W\not=\emptyset$,
\end{itemize}
we have $a(F,W,B_W)>1$.
\end{claim}
\begin{proof}[Proof of Claim \ref{claim: first extraction corresponding to S}]
Suppose not. Since $V_W\subset S_W$, there is a prime $\bb$-divisor $F$ over $X$ such that $F$ is exceptional over $W$, $c_X(F)\subsetneq\bar x$, $c_W(F)\cap S_W\not=\emptyset$ and $a(F,W,B_W)\leq 1$. Since $(W,B_W)$ is $\Qq$-factorial dlt and $S_W$ is an irreducible component of $\lfloor B_W\rfloor$, by Lemma \ref{lem: inv of adj strong}, 
$$\mld(\subset c_W(F)\cap S_W,S_W,B_{S_W})\leq 1.$$
Since $\codim(x,V)\geq1$ and $c_X(F)\subsetneq\bar x$, $c_W(F)\cap S_W\subsetneq g^{-1}(\bar x)\cap S_W$.
Thus there is a prime $\bb$-divisor $F_S$ over $S$, such that $F_S$ that is exceptional over $S_W$, $c_S(F_S)\subsetneq\bar x$ and $a(F_S,S_W,B_{S_W})\leq 1$. This contradicts to the induction hypothesis.
\end{proof}

\noindent\textit{Proof of Lemma \ref{lem: terminalization near dlt} continued}. By Claim \ref{claim: first extraction corresponding to S}, for every prime $\bb$-divisor $F$ over $X$ such that $a(F,X,B)\leq 1$ and $c_X(F)\subsetneq\bar x$, either $F\in\{E_1,\dots,E_m\}$, or $c_W(F)\cap V_W=\emptyset$. In particular, $c_W(F)\not\subset\lfloor B_W\rfloor$. Thus
$$a(F,W,\{B_W\})=a(F,W,B_W)=a(F,X,B)\leq 1.$$
Since $(W,B_W)$ is $\Qq$-factorial dlt, $(W,\{B_W\})$ is klt. By Lemma \ref{lem: klt finite valuation less 1}, there is a finite set of prime $\bb$-divisors $F$ over $X$ such that $a(F,W,\{B_W\})\leq 1$. In particular, there is a finite set $\mathcal{D}'$ of prime $\bb$-divisors $F$ over $X$ such that 
\begin{itemize}
    \item $a(F,X,B)\leq 1$,
    \item $F$ is exceptional over $W$, and
    \item  $c_X(F)\subsetneq\bar x$.
\end{itemize}
Thus $\mathcal{D}=\mathcal{D}'\cup\{E_1,\dots,E_m\}$ is a finite set, which implies (1). (2) follows from (1) and \cite[Corollary 1.4.3]{BCHM10}.

For any birational morphism $f: Y\rightarrow X$ which extracts exactly all the prime $\bb$-divisors belonging to $\mathcal{D}$ such that $K_Y+B_Y=f^*(K_X+B)$, since $(X,B)$ is dlt, (3.a) and (3.b) are immediate. (3.c) follows from the construction of $\mathcal{D}$.
\end{proof}

\section{Complements} The existence of $n$-complements, which was introduced by Shokurov in \cite{Sho92} and proved by Birkar in \cite{Bir19}, plays a key role in the proof of our main theorem. We need a generalized version of $n$-complement in this paper, which is called the \emph{$(n,\Ii)$-complement}.

\begin{defn}[Complements]\label{defn: complement}
	Let $X\to Z$ be a contraction, $B$ an effective $\mathbb R$-divisor on $X$, and $z\in Z$ a point. We say that $(X/Z\ni z,B^+)$ is an \emph{$\Rr$-complement} of $(X/Z\ni z,B)$ if $B^{+}\ge B$, $(X,B^{+})$ is lc and $K_X+B^{+}\sim_{\mathbb R,Z}0$ over an open neighborhood of $z$. 
	
	Let $n>0$ be an integer. An \emph{$n$-complement} of $(X/Z\ni z,B)$ is a pair $(X/Z\ni z,B^+)$, such that over an open neighborhood of $z$,
	\begin{itemize}
		\item $(X,B^+)$ is lc,
		\item  $n(K_X+B^+)\sim_Z 0$, and
		\item $B^+\geq \lfloor B\rfloor+\frac{1}{n}\lfloor (n+1)\{B\}\rfloor$.
	\end{itemize}
	We say that $(X/Z\ni z,B^+)$ is a \emph{monotonic $n$-complement} of $(X/Z\ni z,B)$ if we additionally have $B^+\geq B$.

     If $Z=X$, we may omit $Z$ and say that $(X\ni z,B^+)$ is an $\Rr$-complement (resp. $n$-complement, monotonic $n$-complement) of $(X\ni z,B)$, and in this case, we also say that $(X,B^+)$ is a local $\Rr$-complement (resp. local $n$-complement, monotonic local $n$-complement) of $(X,B)$ near $z$.
\end{defn}

\begin{defn}\label{defn: nI1I2complmenet}
	Let $n>0$ be a integer, $\Ii\subset[0,1]$ a set,  $X\to Z$ a contraction, $B$ an effective $\Rr$-divisor on $X$, and $z\in Z$ a point. An \emph{$(n,\Ii)$-complement} of $(X/Z\ni z,B)$ is an $\Rr$-complement $(X/Z\ni z,B^+)$ of $(X/Z\ni z,B)$ such that $B^+=\sum_i a_iB_i$ satisfying the following.
	\begin{itemize}
	\item Each $a_i\in\Ii$ and $\sum_i a_i=1$, and
	\item each $(X/Z\ni z,B_i)$ is a monotonic $n$-complement of itself.
	\end{itemize}
\end{defn}

The following lemma shows the existence of local $(n,\Ii)$-complements.

\begin{lem}\label{lem: local ni1i2}
Let $d>0$ be an integer and $\Ii_0\subset [0,1]$ a finite set. Then there is an integer $n>0$ and two finite sets $\Ii\subset [0,1]$ and $\Ii'\subset [0,1]\cap\Qq$ depending only on $d$ and $\Ii_0$ satisfying the following. Assume that
\begin{itemize}
    \item $(X\ni x,B)$ is a $\mathbb Q$-factorial lc singularity of dimension $d$, and
    \item $B\in\Ii_0$,
\end{itemize}
then there is an integer $m>0$, real numbers $a_1,\dots,a_m\in (0,1]$, effective $\Qq$-divisors $B_1,\dots,B_m$ and an effective $\Rr$-divisor $G$ on $X$, such that
\begin{enumerate}
\item $a_i\in\Ii$ and $B_i\in\Ii'$ for every $i$,
	\item $\sum_{i=1}^{m} a_i=1$,
	\item $B=\sum_{i=1}^ma_iB_i$, 
	\item $(X\ni x,B_i)$ is lc for every $i$, and
	\item $(X\ni x,B+G)$ is an $(n,\Ii)$-complement of $(X\ni x,B)$.
\end{enumerate}
\end{lem}

\begin{proof} Let $c:=\dim_{\Qq}\Span_{\Qq}(\Ii_0)-1\geq 0$ be an integer and $r_1,\dots,r_c>0$ irrational numbers depending only on $\Ii_0$, such that $1,r_1,\dots,r_c$ are $\Qq$-linearly independent and $\Span_{\Qq}(\Ii_0)=\Span_{\Qq}(1,r_1,\dots,r_c)$. 

We find $m$, $\Ii$ and $\Ii'$ by induction on $c$. When $c=0$, we may let $m=1$, $\Ii:=\{1\}$ and $\Ii':=\Ii_0$.

Suppose that $c\geq 1$. Then there is an integer $n>0$ and $\Qq$-linear functions $s_1,\dots,s_n:\mathbb R^{c+1}\rightarrow\mathbb R$ depending only on $\Ii_0$, such that for every $(X,B)$ as in the assumptions, there are distinct effective Weil divisors $B^1,\dots,B^n$ on $X$ such that $B=\sum_{j=1}^ns_j(1,r_1,\dots,r_c)B^j$. 

By \cite[Theorem 1.6]{Nak16}, we may pick $\epsilon>0$ and $\delta>0$ depending only on $d$ and $\Ii$, such that $r_c+\epsilon$ and $r_c-\delta$ are both rational numbers, and 
$$(X,\sum_{j=1}^ns_j(1,r_1,\dots,r_{c-1},r_c+\epsilon)B^j), (X,\sum_{j=1}^ns_j(1,r_1,\dots,r_{c-1},r_c-\delta)B^j)$$
are both lc. Since
\begin{align*}
K_X+B=&\frac{\delta}{\epsilon+\delta}(K_X+\sum_{j=1}^ns_j(1,r_1,\dots,r_{c-1},r_c+\epsilon)B^j)\\
+&\frac{\epsilon}{\epsilon+\delta}(K_X+\sum_{j=1}^ns_j(1,r_1,\dots,r_{c-1},r_c-\delta)B^j).
\end{align*}
and $s_j(1,r_1,\dots,r_{c-1},r_c+\epsilon)$ and $s_j(1,r_1,\dots,r_{c-1},r_c-\delta)$ 
belong to a finite set of positive real numbers $\Ii_0'$ depending only on $d$ and $\Ii_0$ such that $\dim_{\mathbb Q}\Span_{\mathbb Q}(\Ii_0')\leq c$, by induction on $c$, we get $\Ii$, $\Ii'$, $m$, $a_1,\dots,a_m$ and $B_1,\dots,B_m$ which satisfy (1)-(4).

We only left to find $n$ and $G$ which satisfy (5). By \cite[Theorem 1.7]{Bir19}, there exists an integer $n>0$ depending only on $d$ and $\Ii'$ and effective $\mathbb Q$-divisors $G_1,\dots,G_m$ on $X$, such that
\begin{itemize}
\item each $(X\ni x,B_i+G_i)$ is lc, and
\item each $n(K_X+B_i+G_i)$ is Cartier near $x$.
\end{itemize}
Thus $n$ and $G:=\sum_{i=1}^ma_iG_i$ satisfy our requirements.\end{proof}

\section{Log discrepancies of non-canonical singularities}
In this section we study the structure of log discrepancies of non-canonical singularities, and prove several technical results that are important in the proof of Theorem \ref{thm: alct mld equivalence strict}. 

\subsection{Reduce lc singularities to klt singularities}
It is possible that the minimal log discrepancy \textbf{at} a point is not equivalent to the minimal log discrepancy \textbf{near} a point. 
\begin{ex}
\begin{enumerate}
    \item Let $X:=\mathbb P^2$, $H$ a curve of degree $1$ on $X$ and $x\in H$ a closed point. Then $\mld(x,X,H)=1$ but the the total log discrepancy of $(X,H)$ near $x$ is $0$.
    \item Let $X:=\mathbb P^3$, $B:=H_1+H_2+H_3$ where $H_1,H_2,H_3$ are three hyperplanes of degree $1$ intersecting at a closed point $x$, and $l$ a line such that $x\in l$ and $l\not\subset H_1\cup H_2\cup H_3$. Then $\mld(x,X,B)=0$ but $\mld(\eta_l,X,B)=2$  where $\eta_l$ is the generic point of $l$.
    \item Let $X:=\mathbb P^3$, $B:=\frac{2}{3}(H_1+H_2+H_3)$ where $H_1,H_2,H_3$ are three hyperplanes of degree $1$ intersecting at a line $l$, and $x\in l$ a closed point. Then $\mld(x,X,B)=1$, but $\mld(\eta_l,X,B)=0$ where $\eta_l$ is the generic point of $l$.
\end{enumerate}
\end{ex}

The next lemma shows a connection between the mlds of lc singularities and the mlds of klt singularities.
\begin{lem}\label{lem: lc mld to klt mld}
Let $d>0$ be an integer, $\Ii\subset [0,1]$ a set of real numbers, and $(X\ni x,B)$ an lc singularity of dimension $d$, such that
\begin{itemize}
    \item $B\in\Ii$, and
    \item $0<\mld(x,X,B)\leq 1$,
\end{itemize} 
then 
\begin{enumerate}
\item either there is a $\Qq$-factorial klt singularity $(X'\ni x',B')$ of dimension $d$, such that $B'\in\Ii\cup\{1\}$ and $\mld(x',X',B')=\mld(x,X,B)$, or
\item there is a dlt singularity $(S\ni s,B_S)$ of dimension $d-1$ such that $\mld(s,S,B_S)=\mld(x,X,B)$ and $B_S\in D(\Ii)$.
\end{enumerate}
\end{lem}

\begin{proof}
We may assume that $\codim(x,X)\geq 2$. Let $f: Y\rightarrow X$ be a dlt modification near $x$ and $K_{Y}+B_Y:=f^*(K_{X}+B)$. Suppose that $E$ is a prime $\bb$-divisor over $X$ such that $\mld(x,X,B)$ is attained at $E$. Since $a(E,X,B)>0$, $E$ is exceptional over $Y$. We let $y$ be the generic point of $c_Y(E)$, then 
$$\mld(x,X,B)\leq\mld(y,Y,B_Y)\leq a(E,Y,B_Y)=a(E,X,B)=\mld(x,X,B).$$

Possibly replacing $\Ii$ with $\Ii\cup\{1\}$ and $(X\ni x,B)$ with $(Y\ni y,B_Y)$, we may assume that $(X\ni x,B)$ is $\Qq$-factorial dlt. Possibly removing all the irreducible components of $B$ which do not contain $x$, we may suppose that all the irreducible components of $B$ contain $x$. 

If $(X\ni x,B)$ is klt, we get (1). Therefore we may assume that $(X\ni x,B)$ is not klt. Thus $x\in\lfloor B\rfloor$. Assume that $S_1,\dots,S_n$ are all the irreducible components of $\lfloor B\rfloor$ and $V:=\cap_{i=1}^nS_i$. Since $\mld(x,X,B)>0$, $\bar x\not=V$. Therefore $\codim(x,V)\geq 1$. By Lemma \ref{lem: terminalization near dlt}, there exists an extraction $g: X'\rightarrow X$ which extracts exactly all the prime $\bb$-divisors $F$ such that $c_X(F)\subsetneq\bar x$ and $a(F,X,B)\leq 1$. Since $(X\ni x, B)$ is $\Qq$-factorial dlt, $g$ is an isomorphism near the generic point of $\bar x$. 

Let $x'$ and $B'$ be the strict transforms of $x$ and $B$ on $X'$ respectively, and suppose that 
$$K_{X'}+B'+F':=g^*(K_X+B)$$
where $F'$ is exceptional over $X$. Then $\bar x'\not\subset\Supp F'$. Therefore, for any prime $\bb$-divisor $F$ that is exceptional over $X'$ such that $c_{X'}(F)\subsetneq\bar x'$, 
$$a(F,X',B')\geq a(F,X',B'+F')=a(F,X,B)>1.$$
Let $S$ be the strict transform of $S_1$ on $X'$ and let-
$$K_{S}+B_S:=(K_{X'}+B')|_{S}.$$
By adjunction, $B_{S}\in D(\Ii)$. By Lemma \ref{lem: inv of adj strong},
\begin{align*}
    1&\geq\mld(x,X,B)=a(E,X,B)=\mld(x',X',B'+F')=\mld(x',X',B')\\
    &=\mld(\subset x',X',B')=\mld(\subset x',S,B_S)=\mld(s,S,B_S)
\end{align*}
for some $s\in\bar x'$. Thus $(S\ni s,B_S)$ satisfies (2).
\end{proof}

\subsection{Multiplicities of exceptional divisors with small log discrepancies}

In this subsection, we prove several results on the multiplicities of exceptional divisors with log discrepancies between $0$ and $1$.

\begin{lem}\label{lem: divisor innermultiplicity less than 1}\rm
Assume that 
\begin{itemize}
    \item $n\geq 2$ is an integer,
    \item $(X,B)$ is a dlt pair,
    \item $\alpha_1,\dots,\alpha_n>0$ are real numbers,
    \item $E_1,\dots,E_n$ are exceptional prime $\bb$-divisors over $X$ such that $0<a(E_j,X,B)<1$ for every $1\leq j\leq n$, and
    \item $h: Z\rightarrow X$ is an extraction of $E_1,\dots,E_n$.
\end{itemize}
Then there exists an integer $1\leq i\leq n$ and two birational morphisms $f: Y\rightarrow X$ and $g: W\rightarrow Y$ satisfying the following.
\begin{enumerate}
\item $f$ is an extraction which exactly extracts $E_1,\dots, E_{i-1},E_{i+1},\dots, E_n$, 
\item $g$ is the extraction of $E_i$, and
\item $\mult_{E_i}\sum_{j\not=i}\alpha_jE_{j,Y}<\alpha_i$, where $E_{j,Y}$ is the center of $E_j$ on $Y$ for every $j\not=i$.
\end{enumerate}
\end{lem}

\begin{proof} Let $K_Z+B_Z:=f^*(K_X+B)$. Since $0<a(E_j,X,B)<1$ for every $1\leq j\leq n$, there is a real number $0<\epsilon\ll 1$ such that $(Z,B_Z+\epsilon\sum_{j=1}^n\alpha_jE_j)$ is $\Qq$-factorial dlt. By \cite[Theorem 1.8]{Bir12}, we may run a $(K_Z+B_Z+\epsilon\sum_{j=1}^n\alpha_jE_j)$-MMP$/X$ with scaling. Assume that
\begin{itemize}
    \item $Z\dashrightarrow W$ is the first sequence of flips in this MMP. 
    \item $g: W\rightarrow Y$ is the first divisorial contraction of this MMP,
    \item $i$ is the index such that $g$ contracts $E_i$, and
    \item $f: Y\rightarrow X$ is the induced morphism.
\end{itemize}
We show that $i,f$ and $g$ as above satisfy our requirements. (1)(2) are immediate. To clarify our following statements, we let $E_{j,W}$ be the center of $E_j$ on $W$ for every $1\leq j\leq n$. Since a $(K_Z+B_Z+\epsilon\sum_{j=1}^n\alpha_jE_j)$-MMP$/X$ is also a $(\sum_{j=1}^n\alpha_jE_j)$-MMP$/X$, $g$ is $(\sum_{j=1}^n\alpha_jE_{j,W})$-negative, which implies (3).
\end{proof}

\begin{lem}\label{lem: extremal alc}
Let  $d>0$ be an integer, $\Ii\subset [0,1]$ a set of real numbers, $0<a'<a<1$ two real numbers, and $(X\ni x,B)$ a $\Qq$-factorial klt singularity of dimension $d$, such that
\begin{itemize}
    \item $B\in\Ii$, and
    \item $\mld(x,X,B)=a'$.
\end{itemize}
Then there exists a $\Qq$-factorial klt singularity $(Y\ni y,B_Y)$ of dimension $d$ and a prime $\bb$-divisor $E$ over $Y\ni y$ satisfying the following.
\begin{enumerate}
    \item $B_Y\in\Ii\cup\{1-a\}$,
    \item $a'\leq a(F,Y,B_Y)<a$, and
    \item for any prime $\bb$-divisor $F\not=E$ over $Y\ni y$, $a(F,Y,B_Y)\geq a$.
\end{enumerate}
\end{lem}
\begin{proof}
By Lemma \ref{lem: klt finite valuation less 1}, there are finitely many prime $\bb$-divisors $E_1,\dots,E_n$ over $X\ni x$ such that $a'\leq a_i:=a(E_i,X,B)<a$ for each $i$. If $n=1$, let $(Y\in y,B_Y):=(X\ni x,B)$ and $E:=E_1$ then we are done. Otherwise, By \cite[Corollary 1.4.3]{BCHM10} and Lemma \ref{lem: divisor innermultiplicity less than 1}, possibly reordering the indices of $E_i$, there exists an extraction $g: Y\rightarrow X$ of $E_1,\dots,E_{n-1}$, such that
$$K_Y+B_Y+\sum_{i=1}^{n-1}(a-a_i)E_i=g^*(K_X+B)$$
where $B_Y:=g^{-1}_*B+\sum_{i=1}^{n-1}(1-a)E_i$, and
$$\mult_{E_n}\sum_{i=1}^{n-1}(a-a_i)E_i<a-a_n.$$
Let $E:=E_n$ and $y$ be the generic point of $c_Y(E)$. Since
\begin{align*}
a'\leq a_n&=a(E,Y,B_Y+\sum_{i=1}^{n-1}(a-a_i)E_i)\\
&\leq a(E,Y,B_Y+\sum_{i=1}^{n-1}(a-a_i)E_i)+\mult_{E}\sum_{i=1}^{n-1}(a-a_i)E_i\\
&<a(E,X,B)+(a-a_n)\leq a
\end{align*}
and $$a(E,Y,B_Y)=a(E,Y,B_Y+\sum_{i=1}^{n-1}(a-a_i)E_i)+\mult_{E}\sum_{i=1}^{n-1}(a-a_i)E_i,$$
$(Y\ni y,B_Y)$ and $E$ satisfy our requirements.
\end{proof}

\begin{lem}\label{multiplicity compare}
Assume that 
\begin{itemize}
    \item $(X,B)$ is a klt pair,
    \item $E_1,\dots,E_m$ are exceptional prime $\bb$-divisors over $X$ such that $0<a(E_i,X,B)<1$ for each $i$,
    \item $g: W\rightarrow X$ is an extraction of $E_1,\dots,E_m$, and
    \item $f_i:X_i\rightarrow X$ is the extraction of $E_i$ for every $i$.
\end{itemize} 
To clarify our following statements, we let $E_{i,X_i}$ and $E_{i,W}$ be the centers of $E_i$ on $X_i$ and $W$ respectively. Then for every prime $\bb$-divisor $F$ over $W$,
$$\mult_FE_{i,W}\leq\mult_{F}E_{i,X_i}$$
for every $1\leq i\leq m$.
\end{lem}
\begin{proof}
Let $K_W+B_W:=g^*(K_X+B)$. For every $1\leq i\leq m$, since $a(E_i,X,B)<1$, we may run a $(K_W+B_W-\epsilon E_i)$-MMP over $X$ for some $0<\epsilon\ll 1$. By the uniqueness of the log canonical model, this MMP induces a $(K_W+B_W-\epsilon E_i)$-negative map $h_i: W\dashrightarrow X_i$ over $X$. Thus $h_i$ is $E_i$-positive, and the lemma follows.
\end{proof}

\subsection{A technical lemma}

In this subsection we prove a technical lemma (Lemma \ref{lem: common gap lemma}) for the proof of Theorem \ref{thm: alct mld equivalence strict}. We need the following well-known result on boundedness of number of components:

\begin{thm}[{\cite[Theorem 18.22]{Kol92}}]\label{thm: number of coefficients local}
Let $(X\ni x,\sum b_iB_i)$ be an lc singularity such that $K_X$ and $B_i$ are $\Qq$-Cartier near $x$ and $x\in\cap\Supp B_i$. Then $\sum b_i\le \dim X$.
\end{thm}

\begin{lem}\label{lem: common gap lemma}
Let $d>0$ be an integer and $0<a<a'<1$ two real numbers. Then there is a real number $a<c=c(d,a,a')<a'$ depending only on $d,a$ and $a'$ satisfying the following. Assume that 
\begin{itemize}
    \item $(X\ni x,B)$ is a klt singularity of dimension $d$,
    \item $E$ is a prime $\bb$-divisor over $X\ni x$,
    \item $a(E,X,B)=\mld(x,X,B)<a$, 
    \item $\mathcal{C}$ is the set of all the prime $\bb$-divisors $F\not=E$ over $X\ni x$, such that $a(F,X,B)\leq a'$,
    \item for every $F\in\mathcal{C}$, 
    \begin{itemize}
        \item $f_F:X_F\rightarrow X$ is the extraction of $F$,
        \item $B_{X_F}$ is the strict transform of $B$ on $X_F$, and
        \item $a(E,X_F,B_{X_F}+(1-a')F)\leq a$,
    \end{itemize}
    \item $m>0$ is an integer, $F_1,F_2,\dots F_m\in\mathcal{C}$, such that $a\leq a_i:=a(F_i,X,B)\leq c$ for every $1\leq i\leq m$, and
    \item $g: W\rightarrow X$ is an extraction of $F_1,\dots,F_m$ and $B_W$ is the strict transform of $B$ on $W$,
\end{itemize}
Then $a(E,W,B_W+\sum_{i=1}^m(1-c)F_i)<a$.
\end{lem}

\begin{proof}
We show that $c:=a+\frac{(a'-a)(1-a')}{3d}$ satisfies our requirements. Let $w$ be the generic point of $c_W(E)$. To clarify our following statements, we let $F_{i,W}$ and $F_i'$ be the centers of $F_i$ on $W$ and $X_{F_i}$ respectively. Since $a_i\leq a'$ for every $i$,
$(W,B_W+\sum_{i=1}^m(1-a')F_{i,W})$ is lc near $w$. By Theorem \ref{thm: number of coefficients local}, there are at most $\frac{d}{1-a'}$ different indices $i$ such that $w\in F_{i,W}$. Possibly reordering indices, we may assume that there exists an integer $0\leq n\leq\min\{m,\frac{d}{1-a'}\}$ such that $F_1,\dots,F_n$ are all the $F_i$ such that $w\in F_{i,W}$. In particular,
\begin{align*}
    &a(E,W,B_W+\sum_{i=1}^m(1-c)F_{i,W})\\
    =&a(E,W,B_W+\sum_{i=1}^m(1-a_i)F_{i,W})+\sum_{i=1}^m(c-a_i)\mult_EF_{i,W}\\
    =&a(E,X,B)+\sum_{i=1}^m(c-a_i)\mult_EF_{i,W}\leq a(E,X,B)+\sum_{i=1}^m(c-a)\mult_EF_{i,W}\\
    =&a(E,X,B)+\sum_{i=1}^n(c-a)\mult_EF_{i,W}.\\
\end{align*}
By Lemma \ref{multiplicity compare},
$$\mult_EF_{i,W}\leq\mult_EF_i'$$
for every $1\leq i\leq m$. Since $a(E,X_{F_i},B_{X_{F_i}}+(1-a')F_i')\leq a$ and $a(E,X_{F_i},B_{X_{F_i}}+(1-a_i)F_i')=a(E,X,B)$, 
$$\mult_EF_i'\leq\frac{a-a(E,X,B)}{a'-a_i}\leq\frac{a-a(E,X,B)}{a'-c}$$
for every $1\leq i\leq m$. Thus
$$a(E,X,B)+\sum_{i=1}^n(c-a)\mult_EF_{i,W}\leq a(E,X,B)+n\frac{(c-a)}{a'-c}(a-a(E,X,B)).$$
Since $n\leq\frac{d}{1-a'}$, we deduce that $a(E,X,B)+n\frac{(c-a)}{a'-c}(a-a(E,X,B))<a$. The lemma follows from the inequalities above.
\end{proof}

\section{Proof of the main theorems}

In this section we state several results when assuming the ACC for $a$-lc thresholds or the ACC for mlds. By applying these results, we may reduce Theorem \ref{thm: alct mld equivalence strict} to Theorem \ref{thm: alct implies mld finite coefficient}. By applying results in Section 5, we prove Theorem \ref{thm: alct implies mld finite coefficient}, and deduce all the main theorems.

\begin{lem}\label{lem: acc alct imply lower dimension}
Let $d>1$ be an integer, $a>0$ a real number, and $\Ii\subset [0,1]$ and $\Ii'\subset [0,+\infty)$ two sets. Assume that $\bm{L}(d,\Ii,\Ii',a)$ holds. Then $\bm{L}(d-1,\Ii,\Ii',a)$ holds.
\end{lem}

\begin{proof}
The lemma is immediate by noticing that for any triple $(X_i,B_i;G_i)$ and  point $x_i\in X_i$, $a$-$\lct_{x_i}(X_i,B_i;G_i)=a$-$\lct_{x_i\times\mathbb A^1}(X_i\times\mathbb A^1,B_i\times\mathbb A^1;G_i\times\mathbb A^1)$.
\end{proof}

\begin{prop}\label{prop: finite to dcc coefficient alct mld equivalence}
Let $d>0$ be an integer and $0<a<1$ a real number. Assume that $\bm{M}(d,\Ii_0,a)$ and $\bm{L}'(d,\Ii_0,\Ii_0',a)$ hold for any two finite sets $\Ii_0\subset [0,1]$ and $\Ii_0'\subset [0,+\infty)$. Then $\bm{M}(d,\Ii,a)$ holds for any DCC set $\Ii\subset [0,1]$.
\end{prop}

\begin{proof} Suppose that the proposition does not hold. Then there is
\begin{itemize}
    \item a DCC but not finite set $\Ii\subset [0,1]$,
    \item  a sequence of lc singularities $(X_i\ni x_i,B_i)$ of dimension $d$ such that $B_i\in\Ii$, and
    \item a strictly increasing sequence $a_i:=\mld(x_i,X_i,B_i)$ such that $a=\lim_{i\rightarrow+\infty}a_i$.
\end{itemize}

Possibly replacing $(X_i\ni x_i,B_i)$ with a dlt modification and replacing $\Ii$ with $\Ii\cup\{1\}$, we may assume that $(X_i\ni x_i,B_i)$ is $\mathbb Q$-factorial dlt. Possibly removing all the irreducible components of $B_i$ which do not contain $x_i$, we may assume that all the irreducible components of $B_i$ contain $x_i$. 

Let $\delta:=\min\{t\in\Ii|t>0\}$. Write $B_i=\sum_{j=1}^{m_i}b_{i,j}B_{i,j}$ into its irreducible components. By Theorem \ref{thm: number of coefficients local}, $\sum_{j=1}^m b_{i,j}\leq d$, hence $u_i\leq\frac{d}{\delta}$. Possibly passing to a subsequence, we may assume that 
\begin{itemize}
    \item $m_i=m>0$ is a constant, and
    \item for every $1\leq j\leq m$, $\{b_{i,j}\}_{i=1}^{+\infty}$ is an increasing sequence.
\end{itemize}
Let $b'_j:=\lim_{i\rightarrow+\infty} b_{i,j}$ for each $j$ and
$B'_i:=\sum_{j=1}^{m} b'_jB_{i,j}$
for each $i$. By Theorem \ref{thm: acc lct}, possibly passing to a subsequence, we may assume that $(X_i\ni x_i,B'_i)$ is lc for every $i$. Let $a'_i:=\mld(x_i,X_i,B'_i)$. Possibly passing to a subsequence, we may assume that $a_i'$ has a unique accumulation point $a'$. Since $B'_i\geq B_i$, 
$$0\leq a_i'=\mld(x_i,X_i,B'_i)\leq\mld(x_i,X_i,B_i)=a_i<1-t.$$
 By our assumptions, possibly passing to a subsequence, we may assume that $a'_i$ is decreasing. Thus $a'<a$. Moreover, there is a sequence $\{\alpha_k\}_{k=1}^{+\infty}\subset[0,a)$ of real numbers, such that $\lim_{k\rightarrow+\infty}\alpha_k=a$ and $\{\alpha_k$-$\lct_{x_i}(X_i,0;B_i')\}_{i=1}^{+\infty}$ satisfies the ACC for every fixed $k$.

Pick $k$ such that $a'<\alpha_k<a$. Possibly passing to a subsequence, we may assume that $a_i'<\alpha_k<a_i$ for every $i$. Since $b'_j=\lim_{i\rightarrow+\infty} b_{i,j}$ for each $j$, possibly passing to a subsequence, there exists a strictly increasing sequence of real numbers $\{\beta_i\}_{i=1}^{+\infty}$ such that $\lim_{i\rightarrow+\infty}\beta_i=1$ and 
$$\beta_iB'_i\leq B_i\leq B'_i$$
for every $i$. Thus
$$\beta_i\leq \alpha_k\text{-}\lct_{x_i}(X_i,0;B'_i)<1.$$
Possibly passing to a subsequence, we may assume that  $\alpha_k\text{-}\lct_{x_i}(X_i,0;B'_i)$ is strict increasing, a contradiction.
\end{proof}

\begin{lem}\label{lem: attain implies acc}
Let $d,n>0$ be two integers, $a>0$ a real number, and $\Ii,\Ii'\subset [0,1]$ two finite sets of real numbers. Assume that $\bm{L}(d,\Ii,\Ii'',a)$ holds for any finite set $\Ii''\subset [0,+\infty)$. Assume that
\begin{itemize}
    \item $(X\ni x,B)$ is an lc singularity of dimension $d$,
    \item $(X\ni x, B+G)$ is an $(n,\Ii')$-complement of $(X\ni x,B)$, 
    \item $E$ is a prime $\bb$-divisor over $X\ni x$, 
    \item $0<a$-$\lct_x(X,B;G)\leq 1,$ and
    \item $a$-$\lct_x(X,B;G)$ is attained at $E$, 
\end{itemize}
then $a(E,X,B)$ belongs to an ACC set depending only on $d,n,a,\Ii$ and $\Ii'$.
\end{lem}
\begin{proof}
Let $t:=a$-$\lct_x(X,B;G)$. Then $t$ belongs to an ACC set. Since $(X\ni x, B+G)$ is an $(n,\Ii')$-complement of $(X\ni x,B)$, $a(E,X,B+G)$ belongs to a discrete set. Since $a(E,X,B+G)\leq a(E,X,B+tG)=a$, $a(E,X,B+G)$ belongs to a finite set. Since 
$$a=a(E,X,B+tG)=ta(E,X,B+G)+(1-t)a(E,X,B),$$
we have
$$a(E,X,B)=\frac{a-ta(E,X,B+G)}{1-t}=a(E,X,B+G)+\frac{a-a(E,X,B+G)}{1-t}$$
which belongs to an ACC set.
\end{proof}

\begin{thm}\label{thm: alct implies mld finite coefficient}
Let $d>0$ be an integer and $0<a<1$  a real number. Assume that $\bm{L}'(d,\Ii_0,\Ii_0',a)$ holds for every finite sets $\Ii_0\subset [0,1]$ and $\Ii_0'\subset [0,+\infty)$. Then $\bm{M}(d,\Ii,a)$ holds for every finite set $\Ii\subset [0,1]$.
\end{thm}

\begin{proof}
Suppose that the theorem does not hold. Then there is
\begin{itemize}
    \item a finite set of real numbers $\Ii\subset [0,1]$,
    \item a sequence of lc singularities $(X_i\ni x_i,B_i)$ of dimension $d$ such that $B_i\in\Ii$, and
    \item  a strictly increasing sequence $a_i:=\mld(x_i,X_i,B_i)$ such that $a=\lim_{i\rightarrow+\infty}a_i$.
\end{itemize}

 By Lemma \ref{lem: lc mld to klt mld}, Theorem \ref{thm: alct implies mld finite coefficient} in dimension $d-1$ (when $d>1$), and Lemma \ref{lem: acc alct imply lower dimension}, possibly replacing $\Ii$ with $\Ii\cup\{1\}$ and applying induction on $d$, we may assume that each $(X_i\ni x_i,B_i)$ is $\Qq$-factorial klt.
 
 By Lemma \ref{lem: extremal alc}, possibly replacing $\Ii$ with $\Ii\cup\{1-a\}$, $\{a_i\}_{i=1}^{+\infty}$ with a sequence $\{b_i\}_{i=1}^{+\infty}$ such that $a_i\leq b_i<a$ for every $i$, and passing to a subsequence, we may assume that there is a unique prime $\bb$-divisor $E_i$ over $X_i\ni x_i$ such that $a(E_i,X_i,B_i)<a$. In particular, $a_i=a(E_i,X_i,B_i)$.
 
Since $\bm{L}(d,\Ii,\{0,1\},a)$ holds, let
$$a':=1-\max\{0,t|t<1-a, t\in a\text{-LCT}(d,\Ii,\{0,1\})\},$$
then $a'>a$. Let $c:=c(d,a,a')$ be the number as in Lemma \ref{lem: common gap lemma}. For each $i$, we let
\begin{itemize}
    \item $\mathcal{D}_i:=\{F_i|c_{X_i}(F_i)=\bar x_i, a\leq a(F_i,X_i,B_i)\leq c\}$,
    \item $g_i:W_i\rightarrow X_i$ be an extraction of all the prime $\bb$-divisors in $\mathcal{D}_i$, and
    \item $B_{W_i}$ the strict transform of $B_i$ on $W_i$.
\end{itemize} 
Possibly replacing $\Ii$ with $\Ii\cup\{1-c\}$, $(X_i,B_i)$ with $(W_i,B_{W_i}+(1-c)\sum_{F_i\in\mathcal{D}_i}F_i)$, $x_i$ with the generic point of $c_{W_i}(E_i)$, $\{a_i\}_{i=1}^{+\infty}$ with a sequence $\{b_i\}_{i=1}^{+\infty}$ such that $a_i\leq b_i<a$ for each $i$, and passing to a subsequence, we may assume that $E_i$ is the only prime $\bb$-divisor over $X_i\ni x_i$ such that $a(E_i,X_i,B_i)\leq c$.

By Lemma \ref{lem: local ni1i2}, there is an integer $n>0$ and a finite set $\Ii'\subset [0,1]$ depending only on $d$ and $\Ii$, such that for each $i$, there exists an $(n,\Ii')$-complement $(X_i\ni x_i,B_i+G_i)$ of $(X_i\ni x_i,B_i)$. Possibly passing to a subsequence, we may assume that $\delta:=a(E_i,X_i,B_i+G_i)\leq a_i<a$ is a constant depending only on $d$ and $\Ii$. Let $t:=\frac{c-a}{c}$ and $a'':=a-\frac{(c-a)(a-\delta)}{c}$. Then 
\begin{align*}
    a(E_i,X_i,B_i+tG_i)&=ta(E_i,X_i,B_i+G_i)+(1-t)a(E_i,X_i,B_i)\\
    &<t\delta+(1-t)a=a''
\end{align*}
and
\begin{align*}
    a(F_i,X_i,B_i+tG_i)&=ta(F_i,X_i,B_i+G_i)+(1-t)a(F_i,X_i,B_i)\\
    &>(1-t)c=a>a''
\end{align*}
for every prime $\bb$-divisor $F_i\not=E_i$ over $X_i\ni x_i$. 

By our assumptions, there exists a real number $a''\leq \alpha<a$ such that $\{\alpha$-$\lct_{x_i}(X_i,B_i;G_i)\}_{i=1}^{+\infty}$ is an ACC set. By the inequalities above, $\alpha$-$\lct_{x_i}(X_i,B_i;G_i)$ is attained at $E_i$ for every $i$. By Lemma \ref{lem: attain implies acc}, $\{a_i\}_{i=1}^{+\infty}$ satisfies the ACC, a contradiction.
\end{proof}

\begin{rem}\label{rem: must r divisor}
The essential reason, of why we need to apply the existence of $(n,\Ii)$-complement instead of Birkar's result on the existence of monotonic $n$-complement even if we only consider $\Qq$-pairs, is that $a$ \textbf{may be an irrational number}. In one step of the proof (which corresponds to Lemma \ref{lem: extremal alc}), we need to replace $\Ii$ with $\Ii\cup\{1-a\}$. To explain this more precisely, the reason why we need to replace $\Ii$ with $\Ii\cup\{1-a\}$ is that we require a unique prime $\bb$-divisor over $X$ with log discrepancy $\leq a$ to control the multiplicities as in Lemma \ref{lem: common gap lemma}.

Without showing the existence of $(n,\Ii)$-complement, we cannot deal the case of a sequence of $\Qq$-pairs with a strictly increasing sequence of mlds which converges to an irrational number. In other words, Theorem \ref{thm: alct mld equivalence 01} has established an ``$\Rr$-pair correspondence" between the pairs violating the ACC conjecture for mlds and the pairs violating the ACC conjecture for $a$-lc thresholds, but a ``$\Qq$-pair correspondence" cannot be established in a similar way.
\end{rem}

\begin{proof}[Proof of Theorem \ref{thm: alct mld equivalence strict}]  It follows from Proposition \ref{prop: finite to dcc coefficient alct mld equivalence} and Theorem \ref{thm: alct implies mld finite coefficient}. \end{proof}

\begin{proof}[Proof of Theorem \ref{thm: mld implies alct}] It follows from the same lines of \cite[Proposition 2.1]{BS10} and \cite[Proposition 2.5]{BS10}. For readers' convenience, we give a full proof here.

Suppose that the theorem does not hold. Then there are two DCC sets $\Ii'\subset [0,1]$ and $\Ii''\subset [0,+\infty)$ satisfying the following. For every integer $i>0$, there exists
\begin{itemize}
    \item an lc singularity $(X_i\ni x_i,B_i)$ of dimension $d$ such that $B_i\in\Ii'$,
    \item an $\Rr$-Cartier $\Rr$-divisor $D_i\in\Ii''$ on $X_i$, 
    \item a strictly increasing sequence $t_i:=a$-$\lct_{x_i}(X_i,B_i;D_i)$, and
    \item a real number $t:=\lim_{i\rightarrow+\infty}t_i$.
\end{itemize}
By Theorem \ref{thm: acc lct}, possibly passing to a subsequence, we may assume that $a>0$ and $(X_i\ni x_i, B_i+tD_i)$ is lc for every $i$. We define
\begin{itemize}
    \item $a_i:=\mld(x_i,X_i,B_i)$ and
    \item $t_i':=t_i+\epsilon_i(t-t_i)$.
\end{itemize}
Possibly passing to a subsequence, we may assume that $t_i'$ is strictly increasing. Therefore, the set of the coefficients of $B_i+t_i'D_i$ satisfies the DCC. By the convexity of mlds, 
	\begin{align*}
	a&> \mld(x_i,X_i,B_i+t_i'D_i)\\
	&=\mld(x_i,X_i,\frac{t_i'-t_i}{t-t_i}(B_i+tD_i)+\frac{t-t_i'}{t-t_i}(B_i+t_iD_i))\\
	&\geq\frac{t_i'-t_i}{t-t_i}\mld(x_i,X_i,B_i+tD_i)+\frac{t-t_i'}{t-t_i}\mld(x_i,X_i,B_i+t_iD_i)\\
	&=\frac{t_i'-t_i}{t-t_i}a_i+\frac{t-t_i'}{t-t_i}a=a-\frac{(t_i'-t_i)(a-a_i)}{t-t_i}\\
	&=a-\epsilon_i(a-a_i)\geq (1-\epsilon_i)a.
	\end{align*}
	Possibly passing to a subsequence, we may assume that $\mld(x_i,X_i,B_i+t_i'D_i)$ is a strictly increasing sequence which converges to $a$, a contradiction.
\end{proof}

\begin{proof}[Proof of Theorem \ref{thm: alct mld equivalence 01}] The theorem follows from Theorem \ref{thm: alct mld equivalence strict} and Theorem \ref{thm: mld implies alct}.
\end{proof}

\begin{proof}[Proof of Corollary \ref{cor: alct total mld equivalence}] Since $\tmld(X,B)\leq 1$, the corollary follows from Theorem \ref{thm: alct mld equivalence 01}.\end{proof}

\end{document}